\documentclass[12pt]{article}
\usepackage[margin=1in]{geometry}
\usepackage{amsmath, amssymb, amsthm, setspace}
\usepackage{tikz-cd}
\usepackage{mathrsfs}

\newtheorem{defn}{Definition}[section]
\newtheorem{theorem}{Theorem}[section]

\newtheorem{lemma}[theorem]{Lemma}
\newtheorem{cor}{Corollary}[theorem]
\newtheorem{prop}[theorem]{Proposition}
\newtheorem{ex}[theorem]{Example}

\theoremstyle{remark}
\newtheorem{remark}{Remark}[theorem]

\title{Sofic Lie Algebras}
\author{Cameron Cinel\thanks{Department of Mathematics, University of California, San Diego, La Jolla, CA 92093-0112, USA\ \ \ \ \ \ \ \ \ \ Email: ccinel@ucsd.edu}}
\date{March 3, 2022}

\newcommand{\RR}{\mathbb{R}}

\newcommand{\ep}{\varepsilon}
\newcommand{\CC}{\mathbb{C}}
\newcommand{\al}{\alpha}
\newcommand{\func}[5]{\begin{align*}
	#1:#2&\to#3\\
	#4&\mapsto#5
	\end{align*}}
\newcommand{\wtd}{\widetilde}
\newcommand{\tdth}{\wtd\Theta}
\newcommand{\NN}{\mathbb{N}}
\newcommand{\rk}{\text{rank}}
\newcommand{\gl}{\mathfrak{gl}}
\newcommand{\tr}{\text{tr}}
\newcommand{\SL}{\mathfrak{sl}}
\newcommand{\Id}{\text{Id}}
\newcommand{\spn}{\text{span}}
\newcommand{\hh}{\mathfrak{h}}
\newcommand{\nn}{\mathfrak{n}}
\newcommand{\id}{\text{id}}

\begin{document}
	\maketitle
	
	\begin{abstract}
		We introduce and study soficity for Lie algebras, modelled after linear soficity in associative algebras.
		We introduce equivalent definitions of soficity, one involving metric ultraproducts and the other involving almost representations.
		We prove that any Lie algebra of subexponential growth is sofic.
		We also prove that a Lie algebra over a field of characteristic 0 is sofic if and only if its universal enveloping algebra is linearly sofic.
		Finally, we give explicit families of almost representations for the Witt and Virasoro algebras.
	\end{abstract}
	
	\section{Introduction}
	Sofic groups were first introduced by Misha Gromov \cite{gro99} in relation to the Gottschalk surjunctiviy surjecture conjecture (1976).
	The were later named sofic by Weiss \cite{wie00}.
	We recall the definitions of sofic groups via ultraproducts.
	For more details of ultraproducts in relation to soficity, see \S 2 of \cite{pes08}.
	
	For $n\in\NN$, we denote the symmetric group on $n$ letters by $S_n$.
	We consider the normalized Hamming distance
	\[d_n(\sigma,\tau)=\frac1n|\{i\mid \sigma(i)\neq\tau(i)\}|.\]
	\begin{defn}
		A group $G$ is sofic if there exists and ultrafilter $\omega$ on $\NN$ and a sequence of natural numbers $(n_k)$ such that $G$ embeds into the metric ultraproduct $\prod_{k\to\omega}(S_{n_k},d_{n_k})$.
	\end{defn}
	
	In \cite{ele05}, this definition was reformulated without the notion of ultraproducts.
	It was shown that a group is sofic if and only if it can be locally embedded in some symmetric group almost homomorphically, up to some arbitrarily small error.
	Similar ultraproducts can be constructed from any family of groups equipped with bi-invariant metrics, such as unitary groups over $\CC$ (where subgroups are called hyperlinear).
	
	In particular, for any field $F$, the group $GL_n(F)$ can be endowed the normalized rank metric
	\[\rho_n(A,B)=\frac1n\rk(A-B).\]
	Groups embedable into ultraproducts of $GL_n(F)$ with this metric are called linearly sofic.
	In \cite{arz17} the notion of linear soficity was extended to associative algebras by considering metric ultraproducts of $M_n(F)$.
	In particular, it was shown that a group $G$ is linearly sofic if and only if its group algebra $\CC[G]$ is linearly sofic.
	
	The goal of this paper to do analogous work for Lie algebras.
	Our first section deals with the ultraproduct construction and definition of soficity for Lie algebras as well as basic properties.
	We then provide an equivalent characterization for soficity using almost representations for our Lie algebra.
	
	We also provide examples of sofic Lie algebras using families of almost representations.
	In particular we prove one of our primary results
	\newtheorem*{thm:subexp}{Theorem \ref{thm:subexp}}
	\begin{thm:subexp}
		Any Lie algebra of subexponential growth is sofic.
	\end{thm:subexp}
	We also provide explicit almost representations to show soficity for the Witt and Virasoro algebras.
	This is analogous to soficity in groups and associative algebras, where amenability, and therefore subexponential growth, implies soficity.
	
	Finally, we show that soficity for Lie algebras and associative algebras is compatible.
	\newtheorem*{thm:ueasofic}{Theorem \ref{thm:ueasofic}}
	\begin{thm:ueasofic}
		Let $L$ be a Lie algebra over a field of characteristic 0. Then $L$ is sofic if and only if its universal enveloping algebra $U(L)$ is sofic.
	\end{thm:ueasofic}
	This result should be compared to other work on the growth of Lie algebras, notably the theorem in \cite{smi76} showing that Lie algebras of subexponential growth have universal enveloping algebras of subexponential growth.

	\section{Ultraproducts}
	We define our notion of soficity for Lie algebras using the ultra product construction, similar to that of sofic groups and associative algebras.
	For $n\in\NN$, we call the function \func{\rho_n}{M_n(F)}{[0,1]}{a}{\frac1n\rk(a)} the normalized rank function.
	Let $\omega$ be a non-principal ultrafilter on $\NN$ and $n_k$ be a sequence of natural numbers such that $\lim_{k\to\infty}n_k=\infty$.
	We can extend the $\rho_{n_k}$'s to a function \func{\rho_{n_k}}{M_{n_k}(F)}{[0,1]}{(a_k)}{\lim\limits_{k\to\omega}\rho_{n_k}(a_k).}
	In \cite{arz17}, the pre-image $\rho_\omega^{-1}(0)$ was shown to be an ideal.
	Thus we can construct the ultra-product of the matrix rings, denoted
	\[\prod_{k\to\omega}M_{n_k}(F)/\ker\rho_\omega:=\prod_{k=1}^\infty M_{n_k}(F)/\rho_\omega^{-1}(0).\]
	As an associative algebra, this ultra-product has a natural Lie algebra structure via the commuatator bracket, which we denote
	\[\prod_{k\to\omega}\gl_{n_k}(F)/\ker\rho_\omega:=\left(\prod_{k\to\omega}M_{n_k}(F)/\ker\rho_\omega\right)^{(-)}.\]
	
	\begin{defn}
		We denote the metric ultraproduct of the Lie algebras $\gl_{n_k}(F)$ as $\prod_{k\to\omega}\gl_{n_k}(F)/\ker\rho_\omega$, called a universal sofic Lie algebra.
	\end{defn}
	
	For any subalgebra $L\subset \gl_n(F)$, we can restrict the normalized rank map and hence create different metric ultraproducts.
	
	\begin{lemma}
		For any field $F$,
		\[\prod_{k\to\omega}\gl_{n_k}(F)/\ker\rho_\omega\cong\prod_{k\to\omega}\SL_{n_k}(F)/\ker\rho_\omega.\]
	\end{lemma}
	
	\begin{proof}
		Clearly, we can embed $\prod_{k\to\omega}\SL_{n_k}(F)/\ker\rho_\omega\hookrightarrow\prod_{k\to\omega}\gl_{n_k}(F)/\ker\rho_\omega$.
		Moreover, for any sequence $(A_k)\in\prod_{k\to\omega}\gl_{n_k}(F)/\ker\rho_\omega$, we can consider the second sequence $(A_k-E_{11}\tr(A))$.
		Then
		\[\rho_\omega((A_k)-(A_k-E_{11}\tr(A)))=\lim_{k\to\omega}\frac{1}{n_k}=0\] so $(A_k)=(A_k-E_{11}\tr(A))$ in the ultra-product.
		Thus the natural embedding is a surjection.
	\end{proof}

	We now define the map object of this paper.

	\begin{defn}
		A Lie algebra $L$ over a field $F$ is called linearly sofic if there exists a Lie algebra embedding of $L$ into some metric ultraproduct of some $\gl_{n_k}(F)$'s.
	\end{defn}

	\begin{prop}
		For an algebraically closed field $F$, the associative algebra $\prod_{k\to\omega}M_{n_k}(F)/\ker\rho_\omega$ is simple.
	\end{prop}
	\begin{proof}
		If $\omega$ is a principal ultrafilter, this is trivially true.
		So we assume that $\omega$ is free.
		
		Denote $R:=\prod_{k\to\omega}M_{n_k}(F)/\ker\rho_\omega$.
		Suppose $0\neq x\in R$.
		Let $\delta:=\rho_\omega(x)$.
		Since $\delta>0$, there exists $n\in\NN$ such that $\frac1n<\delta$.
		Thus, if $(B_k)\in\prod M_{n_k}(F)$ is a representative for $x$, there exists some $S\in\omega$ such that $\rho_{n_k}(B_k)\geq\frac1n$ for every $k\in S$.
		
		Let $J_k$ be the Jordan normal form of $B_k$.
		Then $J_k$ row equivalent to 
		\[A^{(m)}_k:=\sum_{i=1}^{\lfloor\frac{n_k}{n}\rfloor}E_{\lfloor\frac{mn_k}{n}\rfloor+i,\lfloor\frac{mn_k}{n}\rfloor+i}\]
		where $0\leq m\leq n-1$.
		Thus, there exists $P_k,Q_k\in M_{n_k}(F)$ such that
		\[B_k=P_kA^{(m)}_kQ_k.\]
		
		Define $C_k\in M_{n_k}(F)$ via
		\[C_k=\begin{cases}
		\sum_{m=0}^{n-1}A^{(m)}_k,& k\in S\\
		0,& k\in S^C
		\end{cases}.\]
		Then if $y=[C_k]\in R$, we have that $y\in (x)$, the ideal generated by $x$.
		Moreover, we notice that for $k\in S$,
		\[\rk(C_k-I_{n_k})\leq n-1.\]
		Let $\ep>0$.
		Since $\omega$ is a free ultrafilter, $S$ must be infinite.
		Therefore, there exists an infinite subset $S'\subset S$ such that $\rho_{n_k}(C_k-I_{n_k})<\ep$.
		Notice that $S'\in\omega$ since otherwise, $(S')^C\cap S\in\omega$ making $\omega$ a principal ultrafilter.
		Therefore $\rho_\omega(y-[I_{n_k}])=0$ and $(x)=R$.
	\end{proof}
		
	\begin{cor}
		For an algebraically closed field $F$, the non-trivial ideals of $\prod_{k\to\omega}\gl_{n_k}(F)/\ker\rho_\omega$ are contained in its center.
	\end{cor}

	\begin{proof}
		This follows from Theorem 2 of \cite{her61}.
	\end{proof}

	\section{Almost representations}
	Just as with sofic groups and linearly sofic associative algebras, linearly sofic Lie algebras can be defined via families of maps that are Lie algebra homomorphisms up to a small error.
	We follow a similar approach to section 11.2 of \cite{arz17}.
	
	\begin{defn}
		Let $L$ be a Lie algebra, $W\subset L$ a finite dimensional subspace, $V$ a finite dimensional vector space, and $\ep>0$.
		A linear map $\varphi:W\to\gl(V)$ is called and $\ep$-almost representation of $W$ if there exists a subspace $V_\ep\subset V$ such that:
		\begin{enumerate}
			\item for all $x,y\in W$ such that $[x,y]\in W$,
			\[\varphi([x,y])|_{V_\ep}=(\varphi(x)\varphi(y)-\varphi(y)\varphi(x))|_{V_\ep};\]
			\item $\dim V$-$\dim V_\ep\leq\ep\cdot\dim V$.
		\end{enumerate}
	\end{defn}

	Clearly a morphism from a Lie algebra to a universal sofic Lie algebra gives rise to a family of almost representations of the Lie algebra.
	Similarly, a family of almost representations on a covering of a Lie algebra gives a morphism to a universal sofic Lie algebra.
	However, not all families of almost representations will correspond to embeddings.
	In particular, some elements of our Lie algebra will have to be mapped to arbitrarily small rank transformations via the family of almost representations.
	This gives rise to to the following subspace of our Lie algebra.
	
	\begin{defn}
		For a Lie algebra $L$, the sofic radical of $L$, denote $SR(L)$, is defined as follows: $p\in SR(L)$ if for every $\delta>0$, there exists a finite dimensional $W\subset L$ containing $p$ and $n_\delta>0$ such that if $0<\ep<n_\delta$ and $\varphi_{W,\ep}:W\to\gl(V)$ is an $\ep$-almost representation, then
		\[\dim\text{Im }\varphi_{W,\ep}(p)<\delta\cdot\dim V.\]
	\end{defn}
	
	The sofic radical is essentially the collection of "bad" elements of our Lie algebra when it comes to trying to make embeddings of it into a universal sofic Lie algebra.
	This view is summarized by the following lemma and corollary.
	
	\begin{lemma}
		For a Lie algebra $L$, $p\in SR(L)$ if and only if for every Lie algebra homomorphism
		\[\Theta:L\to\prod_\omega \gl_{n_k}(F)/\rho_\omega\] we have that $\Theta(p)=0$.
	\end{lemma}
	
	\begin{proof}
		Suppose $p\in SR(L)$ and let $\Theta:L\to\prod_\omega \gl_{n_k}(F)/\rho_\omega$ be a Lie algebra homomorphism with lifts $\theta_k:L\to \gl_{n_k}(F)$.
		Fix $\delta>0$ and choose $n_\delta$ and finite dimensional $W\subset L$ from the definition of the sofic radical.
		Choose $0<\ep<n_\delta$.
		Then there exists $S\in\omega$ such that $\theta_k|_W$ is an $\ep$-almost representation for $W$ for every $k\in S$.
		Therefore $\dim\text{Im }\theta_k(p)<\delta n_k$.
		In other words, $\rho_{n_k}(\theta_k(p))<\delta$ for every $k\in H$.
		Therefore $\rho_\omega(\Theta(p))<\delta$.
		Since $\delta$ was arbitrary, we get that $\Theta(p)=0$.
		
		Now suppose that $p\in L\setminus SR(L)$.
		Then there exists $\delta>0$ such that for any finite dimensional $W\subset L$ containing $p$ and $n>0$, there exists $0<\ep<n$ and $\ep$-almost representation $\varphi:L\to\gl(V)$ such that $\dim\text{Im }\varphi(p)\geq\delta\cdot\dim V$.
		
		Let $W_k$ be an increasing sequence of finite dimensional subspaces of $L$ containing $p$ that cover $L$.
		Then there exists a sequence of $\ep_k$ that converges to 0 and $\ep_k$-almost representations $\theta_k:W_k\to\gl(V_k)$ such that $\dim\text{Im }\theta_k(p)\geq\delta\cdot\dim V_k$.
		Define a map $\hat\Theta:L\to\prod\gl(V_k)$ by $\hat\Theta(x)=(\hat\theta_k(x))$ where
		\[\hat\theta_k(x)=\begin{cases}
		\theta_k(x),& x\in W_k\\
		0,& x\in L\setminus W_k
		\end{cases}.\]
		Choose a non-principal ultrafilter $\omega$ of $\NN$ and let $\Theta:L\to\prod_\omega\gl(V_k)/\rho_\omega$ be the composition of $\hat\Theta$ with the quotient map.
		Then since $\ep_k\to0$, we have that $\Theta$ is a Lie algebra homomorphism and $\rho_\omega(\Theta(p))\geq\delta>0$.
	\end{proof}
	
	From this lemma, we get the following corollary
	\begin{cor}
		For a Lie algebra $L$, $SR(L)$ is an ideal.
		Moreover, $SR(L/SR(L))=(0)$.
	\end{cor}
	
	\begin{proof}
		The lemma show that
		\[SR(L)=\bigcap_{\Theta\in S}\ker\Theta\] where
		\[S=\left\{\Theta:L\to\prod_\omega \gl_{n_k}(F)/\rho_\omega\mid\Theta\text{ is a Lie algebra homomorphism}\right\}\] so it is clear that $SR(L)$ is an ideal.
		
		Now suppose $0\neq p\in L/SR(L)$ and let $q\in L$ be a pre-image of $p$.
		Then $q\notin SR(L)$ so there exists a Lie algebra homomorphism $\Theta:L\to\prod_\omega \gl_{n_k}(F)/\rho_\omega$ such that $\Theta(q)\neq0$.
		Since $SR(L)\subset\ker\Theta$, we can get a map $\hat\Theta$ by composing $\Theta$ with the quotient $L\to L/SR(L)$.
		Then we have that $\hat\Theta(p)=\Theta(q)\neq0$ so $p\notin SR(L/SR(L))$.
	\end{proof}
	
	We can now use the sofic radical to characterize sofic Lie algebras.
	\begin{theorem}
		A Lie algebra $L$ is sofic if and only if $SR(L)=(0)$.
	\end{theorem}
	
	\begin{proof}
		The forward implication is trivial so we only show the reverse implication.
		
		Let $L$ be a Lie algebra such that $SR(L)=0$.
		Then for every $p\in L\setminus(0)$, there exists a Lie algebra homomorphism $\Theta_p:L\to\prod_\omega \gl_{n_{k,p}}(F)/\rho_\omega$ such that $\Theta_p(p)\neq0$.
		
		Let $\{x_i\}_{i\in\NN}\subset L$ be a basis for $L$ as a vector space.
		We shall construct maps $\Psi_m:L\to\prod_\omega \gl_{n_{k,m}}/\rho_\omega$ such that $\ker\Psi_m\cap\spn\{x_1,\dots,x_m\}=(0)$ for every $m\in \NN$.
		
		Let $\Psi_1=\Theta_{x_1}$.
		Now suppose for $m\in\NN_{\geq2}$, we have a map $\Psi_{m-1}$ as above.
		Then 
		\[\dim(\Psi_{m-1}\cap\spn\{x_1,\dots,x_m\})\leq1.\]
		If the dimension is 0, let $\Psi_m=\Psi_{m-1}$.
		Otherwise, choose a non-zero element $y_m$ in the intersection.
		In this case, we define $\Psi_m=\Psi_{m-1}\oplus\Theta_{y_m}$.
		Then if $z\in\ker\Psi_m\cap\spn\{x_1,\dots,x_m\}$, we have that $z\in\ker\Psi_{m-1}\cap\spn\{x_1,\dots,x_m\}$.
		Thus $z=\al y_m$ for some $\al\in F$.
		However, $z\in\ker\Theta_{y_m}$ so $z=0$.
		
		Now we construct a map $\Phi:L\to\prod_\omega \gl_{n_k}(F)/\rho_\omega$ such that $\ker\Phi\subset\bigcap\ker\Psi_m$.
		Let $n_k=n_{k,1}\cdots n_{k,k}$.
		Tensor each component of $\Psi_m$ with an appropriately sized identity matrix gives us maps $\hat\Psi_m:L\to\prod_\omega \gl_{\hat{n}_{k,m}}(F)/\rho_\omega$ where $\hat{n}_{k,m}=n_k$ if $m\leq k$ and $n_{k,m}$ otherwise.
		Let $\hat\psi_{k,m}:L\to \gl_{n_k}(F)$ be a lift of $\hat\Psi_m$ for $m\leq k$.
		
		Define a map $\varphi_k:A\to \gl_{2^kn_k}$ by
		\[\varphi_k=(\hat\psi_{k,1}\otimes\Id_{2^{k-1}}\oplus(\hat\psi_{k,2}\otimes\Id_{2^{k-2}})\oplus\cdots\oplus(\hat\psi_{k,k}\otimes\Id_1)\oplus\Id_{n_k}.\]
		Let $\Phi=\prod_\omega\varphi_k/\rho_\omega$.
		
		A direct calculation gives us that for any $x\in L$
		\begin{align*}
		\rho_\omega(\Phi(x))&=\lim_{k\to\omega}\frac{\rk(\varphi_k(x))}{2^kn_k^k}\\
		&=\lim_{k\to\omega}\frac{1}{2^kn_k^k}\sum_{i=1}^k\rk((\hat\psi_{k,1}\otimes\Id_{2^{k-1}})(x))\\
		&=\lim_{k\to\omega}\frac{1}{n_k^k}\sum_{i=1}^k\frac{\rk(\hat\psi_{k,1}(x))}{2^i}\\
		&=\lim_{k\to\omega}\sum_{i=1}^k\frac{\rk(\psi_{k,1}(x))}{2^in_{k,i}}\\
		&=\sum_{i=1}^k\frac{1}{2^i}\rho_\omega(\Psi_i(x))
		\end{align*}
		where $\psi_{k,m}:L\to \gl_{n_{k,m}}$ is a lift of $\Psi_m$.
		
		Therefore $\ker\Phi\subset\bigcap\ker\Psi_m=(0)$, so we have that $\Phi$ is injective.
	\end{proof}

	\section{Examples of Sofic Lie Algebras}
	By using the sofic radical, we can determine if particular Lie algebras is sofic.
	\begin{prop}
		Any abelian Lie algebra is sofic.
	\end{prop} 
	
	\begin{proof}
		Suppose $L$ is an abelian Lie algebra over a field $F$ and suppose that $0\neq p\in L$.
		Then for any finite dimensional subspace $p\in W\subset L$ and $\ep>0$, there exists an $\ep$-almost representation $\varphi:W\to F\cong\gl_1(F)$ such that $\varphi(p)=1$ and $\varphi(W\setminus Fp)=0$.
		Thus $\dim\text{Im}\varphi(p)=1=\dim F$ so $p\notin SR(L)$.
	\end{proof}

	For groups, those of subexponential growth are amenable and therefore sofic.
	We show the same result holds true for Lie alebras.
	
	\begin{theorem}
		\label{thm:subexp}
		Any Lie algebra of subexponential growth is sofic.
	\end{theorem}
	
	\begin{proof}
		Let $L$ be a Lie algebra of subeponential growth generated by the set $X\subset L$.
		Let $V_n$ denote the words in $L$ of length at most $n$.
		We inductively create a basis for $L$ as follows.
		Let $\{x_1,\dots,x_{\dim V_1}\}$ be a basis for $V_1$.
		For $n\geq 2$, let $\{x_{\dim V_{n-1}+1},\dots,x_{\dim V_n}\}$ be the preimage in $L$ of a basis for $V_n/V_{n-1}$.
		We also define $W_n$ to be the subspace of $U(L)$ spanned by words of length at most $n$ and $\gamma(n)=\dim W_n$.
		
		For $m>n$, we define a linear map $\varphi_{n,m}:V_n\to \gl(W_m)$ where 
		\[\varphi_{n,m}(x_i)(x_{j_1}\cdots x_{j_k})=\begin{cases}
		x_ix_{j_1}\cdots x_{j_k},& x_ix_{j_1}\cdots x_{j_k}\in W_m\\
		0,& \text{otherwise}
		\end{cases}.\]
		We notice on that for $v\in W_{m-n}$ and $x,y\in V_n$,
		\[[\varphi_{n,m,k}(x),\varphi_{n,m,k}(y)](v)=\varphi_{n,m,k}([x,y])(v).\]
		
%
		Let us show that for a fixed $n$ that
		\[\lim_{m\to\infty}\frac{\gamma(m-n)}{\gamma(m)}=1.\]
		Since $L$ is of subexponential growth, we have that $\gamma$ is a function of subexponential growth from \cite{smi76}.
		To see this suppose that $f:\RR\to\RR$ is a function of subexponential growth and $d>0$.
		We have that if 
		\[\lim_{x\to\infty}\frac{f(x)}{f(x+d)}=\frac{1}{L}<1\]
		Then there exists $x_0\in\RR$ such that
		\[f(x_0+nd)\geq L^nf(x_0).\]
		Therefore $f(x)\geq L^\frac{x-x_0}{d-1}f(x_0)$ and is therefore not of subexponential growth.
		
		Thus for any $n\in\NN$ and $\ep>0$, we can find $\varphi_{n,m}$ such that 
		\[\frac{\dim W_m-\dim W_{n-m}}{\dim W_m}<1-(1-\ep)=\ep.\]
		
		Now suppose $p\in L\setminus\{0\}$, $V\subset L$ is a finite dimensional subspace containing $p$, and $\ep>0$.
		Then there exists $n\in\NN$ such that $V\subset V_n$.
		Thus for a sufficiently large $m\in\NN$, we have an $\ep$-almost representation $\psi:=\varphi_{n,m}|_{V}$ for $V$.
		On $W_{m-n}$, we have that $\psi(p)$ works like left multiplication by $p$, considered as an element of $U(L)$.
		Thus $\dim\text{Im}(\psi(p))\geq\dim W_{m-n}$.
		By our choice of $m$ to make $\psi$ an $\ep$-almost representation, we have that
		\[\dim W_{m-n}>(1-\ep)\dim W_m.\]
		Hence the normalized rank of $\psi(p)$ is greater than $1-\ep$.
		
		Thus for any non-zero $p\in L$, any finite dimensional subspace $p\in V\subset L$, and $\ep>0$, we can find an $\ep$-almost representation $\psi$ for $V$ such that the normalized rank of $\psi(p)$ is at least 1/2.
		Specifically, we can choose $\psi$ to be a $\min\{\ep,1/2\}$-almost representation.
		Therefore $p\notin SR(L)$ so $SR(L)=(0)$ and $L$ is sofic.
	\end{proof}

	We now provide explicit families of almost representations for two particular Lie algebras.
	Though both of the examples are of subexponential growth, their families of almost homomorphisms are not based on their subspaces of words of particular length.
	
	\begin{ex}
	The Witt algebra $L_W$, which is the Lie algebra of derivations for $\CC[t,t^{-1}]$, is sofic.
	\end{ex}
	We can cover $L_W$ by the finitely dimensional subspaces
	\[V_n=\sum_{i=-n}^n\CC x_i\] where $x_i=-t^{n+1}\frac{d}{dt}$.
	We also consider the finite dimensional vector spaces
	\[W_n=\sum_{i=-n}^n\CC t^i\subset\CC[t,t^{-1}].\]
	For $m\geq n$, define a linear map $\varphi_{n,m}:V_n\to\gl(W_m)$ via
	\[\varphi_{n,m}(x_i)(t^j)=\begin{cases}
	-jt^{i+j},& j\leq m-i\\
	0,& \text{otherwise}
	\end{cases}\]
	and
	\[\varphi_{n,m}(x_{-i})(t^j)=\begin{cases}
	-jt^{j-i},& j\geq i-m\\
	0,&\text{otherwise}
	\end{cases}\] for $i\geq0$.
	For any $n\in\NN$ and $\ep>0$, $m$ can be made sufficiently large so as to make $\varphi_{n,m}$ and $\ep$-almost homomorphism.
	
	\begin{ex}
		The Virasoro algebra $L_V$, which is the unique central extension of $L_W$, is sofic.
	\end{ex}

	To construct the family of almost representations, we need to use the Verma modules of $L_V$.
	For more information on the Verma modules, see \cite{wak86}.
	
	We have a basis of $L_V$ given by the $x_i$'s from $L_W$ and a central element $c$.
	Define the subspaces $\hh=\CC x_0+\CC c$ and 
	\[\nn_+=\sum_{k=1}^\infty\CC x_k.\]
	Given $\lambda\in\hh^*$, we define the Verma module to be the space
	\[M(\lambda)=U(L_V)\otimes_{U(\hh\oplus\nn_+)}\CC\] where the action of $\hh\oplus\nn_+$ on $\CC$ is given by
		\[(h+x)\cdot\al=\lambda(h)\al.\]
	For $m,d\in\NN$, we consider the subspaces
	\[M(\lambda)_{m,d}=\text{span}\{x_{-i_1}\cdots x_{-i_k}\otimes1\mid0\leq k\leq r, 0\leq i_k\leq\cdots\leq i_1\leq m\}.\]
	We then define linear maps $\varphi_{n,m,d}:V_n+\CC c\to\gl(M(\lambda)_{m,d})$ via $\varphi_{n,m,d}(c)=\id_{M(\lambda)_{m,d}}$ and
	\[\varphi_{n,m,d}(x_r)(x_{-i_1}\cdots x_{-i_k})=\begin{cases}
	0, & k=d\text{ or }r\geq i_1-m\\
	x_kx_{-i_1}\cdots x_{-i_k},& \text{otherwise}
	\end{cases}.\]
	Just as in the case with $L_W$, for a fixed $n,d\in\NN$ and $\ep>0$, we can choose $m$ large enough to make $\varphi_{n,m,d}$ an $\ep$-almost representation.

	\section{Soficity of Universal Enveloping Algebras}
	In \cite{arz17}, it was shown that a group is linearly sofic if and only if its corresponding group algebra is itself linearly sofic.
	For Lie algebras, we have a similar construction with universal enveloping algebras.
	This section is spent on showing equivalence of soficity for Lie algebras and universal enveloping algebras.
	
	We first need the following result from \cite{arz17}.
	\begin{prop}
		Let $A$ be a unital algebra and $\omega\subset\mathscr{P}(\NN)$ an ultrafilter.
		If we have a set $\{\Theta^i\}_{i\in I}$ of unital algebra morphisms from $A$ to $\prod_\omega M_{n_{k,i}}(F)/\rho_\omega$, then there exists a unital algebra homomorphism $\psi:A\to\prod_\omega M_m(F)/\rho_\omega$ such that \[\ker\psi\subset\bigcap_{i\in I}\ker\Theta^i.\]
	\end{prop}

	We can now tackle our main result:
	\begin{theorem}
		\label{thm:ueasofic}
		Let $L$ be a Lie algebra over a field $F$ of characteristic 0.
		Then $L$ is sofic if and only if $U(L)$ is sofic.
	\end{theorem}
	\begin{proof}
		Since $L$ embeds in $U(L)$ the reverse direction is trivial.
		Thus we only prove the forward direction.
		
		Since $L$ is sofic, there exists a Lie algebra homomorphism
		\[\Theta:L\to\prod_\omega \gl_{n_k}(F)/\rho_\omega=:A.\]
		Let $\Theta_1=\Theta$ and 
		\[\Theta_j=\underbrace{1\otimes\cdots\otimes1}_{j-1\text{ times}}\otimes\Theta\]
		for $j\geq2$.
		Define a new Lie algebra homomorphism $\Theta^i:L\to\bigotimes_{j=1}^iA$ by $\Theta^0=0$, $\Theta^1=\Theta$, and $\Theta^i=\Theta^{i-1}\otimes 1+1\otimes\Theta_{i-1}$.
		Since finite tensor product of matrix rings are again matrix rings, we use the proposition to get a unital algebra homomorphism
		\[\psi:U(L)\to\prod_{\omega}{M_m(F)}/\rho_\omega\] such that
		\[\ker\psi\subset\bigcap_{i=0}^\infty\tdth^i.\]
		We claim that $\psi$ is injective.
		
		Let $\{x_j\}_{j\in J}\subset L$ be a basis.
		First suppose that $x_{j_1}\cdots x_{j_d}\in\ker\psi$ is a monomial for $d\geq2$.
		Then by construction
		\[\tdth^i(x_{j_1}\cdots x_{j_d})=0\] for all $i\geq1$.
		Notice that $\Theta^d(x_{j_1}\cdots x_{j_d})$ will consist of a linear sum of pure tensors of $1$ and products of the $\Theta(x_{j_i})$'s.
		In fact, we have two sets of pure tensors: one set consisting of those with at least one 1 and at least one degree 2 or greater monomial and the other set consisting of pure tensors with only linear monomials and no 1's, i.e. pure tensors of the form $x_{j_{\sigma(1)}}\otimes\cdots\otimes x_{j_{\sigma(d)}}$ for some $\sigma\in S_d$.
		In fact we get exactly one such pure tensor for every $\sigma\in S_d$.
		Call the first set the lower order tensors and the second set the high order tensors.
		
		Now notice that $\tdth^{d-1}(x_{j_1}\cdots x_{j_d})$ is a sum of low order tensors of $\tdth^d(x_{j_1}\cdots x_{j_d})$ with one of the 1's removed.
		Thus, viewing $0$ as $0\otimes0\otimes\cdots\otimes0$ ($d-1$ times), we can tensor the equation
		\[\tdth^{d-1}(x_{j_1}\cdots x_{j_d})=0\] by 1 in $d$ spots and adding up all the new equations shows that the low order tensors of $\tdth^d(x_{j_1}\cdots x_{j_d})$ sum to 0.
		Therefore, we get that
		\[\sum_{\sigma\in S_d}\Theta(x_{j_{\sigma(1)}})\otimes\cdots\otimes \Theta(x_{j_{\sigma(d)}})=0.\]
		
		Now suppose that $i>d$.
		Then we can see that in the sum of pure tensors in $\tdth^i(x_{j_1}\cdots x_{j_d})$, every pure tensor has at least one 1 in it.
		Thus repeating the same process of the tensor the equation $\tdth^{i-1}(x_{j_1}\cdots x_{j_d})=0$ with $1$ in $i$ spots and summing up, we get back to the equation $\tdth^i(x_{j_1}\cdots x_{j_d})=0$.
		
		Now we are ready to prove injectivity of $\psi$.
		Suppose that $a\in\ker\psi\setminus\{0\}$.
		Put an order on the indexing set $J$ of the basis for $L$.
		Then by the Poincar\'e-Birkoff-Witt Theorem, we can uniquely write $a$ as
		\[a=\sum\limits_{i=1}^n\al_ix_{j^{(i)}_1}^{e^{(i)}_1}\cdots x_{j^{(i)}_{d_i}}^{e^{(i)}_{d_i}}+\al_0\] where $\al_i\in F$, $\al_i\neq0$ for $i\geq1$, $e^{(i)}_j\geq1$ for every $i,j$, $d_i\geq1$, and $j^{(i)}_k< j^{(i)}_{k+1}$.
		By applying $\tdth^0$, we get rid of every term but $\al_0$, so we see that $\al_0=0$.
		Let 
		\[D_i=\sum\limits_{k=1}^{d_i}e^{(i)}_k\] be the degree of the $i$-th monomial and let $D=\max\{D_i\}_{1\leq i\leq n}$.
		Since we know that $\Theta$ is injective on $L$, we get a contradiction if $D=1$.
		Thus assume that $D\geq2$.
		
		Repeating the same process we did on monomials, take the equation $\tdth^{D-1}(a)=0$, tensor it with 1 in the $D$ possible spots and add up all the equations.
		Then we subtract that sum from the equation $\tdth^{D}(a)=0$.
		For every $i$ such that $D_i<D$, the terms coming from the $i$-th monomial will vanish in the new equation.
		Thus we are on left with the monomials of degree exactly $D$.
		The equation we are left with is a linear combination of pure tensors of the form $\Theta(x_{j_1})\otimes\cdots\otimes\Theta(x_{j_D})$.
		Since our monomials are ordered, the only way to have such pure tensors are equal is if they come from the same original monomial.
		Therefore we have a non-trivial linear combination.
		However, $\{\Theta(x_j)\}_{j\in J}\subset A$ is linearly independent since $\Theta$ is injective, so all the weights in the linear combination must be 0.
		Since $F$ is characteristic, we get that $\al_i=0$ for every $i$ such that $D_i=D$.
		This is a contradiction so no such $a$ exists and $\ker\psi=(0)$.
	\end{proof}
	
	\begin{remark}
		This equivalence of soficity is particularly useful when combined with Proposition 11.6 from \cite{arz17} stating amenable algebras without zero divisors are sofic.
		A particular application is a simplification of Theorem \ref{thm:subexp} for Lie algebras over fields of characteristic 0.
		Indeed, if $L$ is such a Lie algebra, by Theorem 7 of \cite{smi76}, $U(L)$ is also of subexponential growth.
		From \cite{ele03}, we see that $U(L)$ is amenable algebra with no zero divisors.
		Therefore $U(L)$, and thus $L$, is sofic.
	\end{remark}
	
	\newpage
	\bibliographystyle{alpha}
	\bibliography{SoficPaperbib}

\begin{thebibliography}{WY86}

\bibitem[AP17]{arz17}
Goulnara Arzhantseva and Liviu Pa\u{u}nescu.
\newblock Linear sofic groups and algebras.
\newblock {\em Transactions of The American Mathematical Society},
  369(4):2285--2310, April 2017.

\bibitem[Ele03]{ele03}
G\'abor Elek.
\newblock The amenability of affine algebras.
\newblock {\em Journal of Algebra}, 264:469--478, 2003.

\bibitem[ES05]{ele05}
G{\'{a}}bor Elek and Endre Szab{\'{o}}.
\newblock Hyperlinearity, essentially free actions and l2-invariants. the sofic
  property.
\newblock {\em Mathematische Annalen}, 332(2):421--441, apr 2005.

\bibitem[Gro99]{gro99}
Misha Gromov.
\newblock Endomorphism of symbolic algebraic varieties.
\newblock {\em Journal of the European Mathematical Society}, 1(2):109--197,
  1999.

\bibitem[Her61]{her61}
Isreal Herstein.
\newblock Lie and jordan structures in simple associative algebras.
\newblock {\em Bulletin of the American Mathematical Society}, 67(6):517--531,
  1961.

\bibitem[Pes08]{pes08}
Vladimir Pestov.
\newblock Hyperlinear and sofic groups: a brief guide.
\newblock {\em Bulletin of Symbolic Logic}, 14(4):449--480, December 2008.

\bibitem[Smi76]{smi76}
Martha~K. Smith.
\newblock Universal enveloping algebras with subexponential but not
  polynomially bounded growth.
\newblock {\em Proceedings of the American Mathematical Society}, 60:22--24,
  October 1976.

\bibitem[Wei00]{wie00}
Benjamin Weiss.
\newblock Sofic groups and dynamical systems.
\newblock {\em The Indian Journal of Statistics}, 62(3):350--359, 2000.

\bibitem[WY86]{wak86}
Minoru Wakimoto and Hirofumi Yamada.
\newblock The fock representations of the virasoro algebra and the hirota
  equations of the modified kp hierarchies.
\newblock {\em Hiroshima Math Journal}, 16:427--441, 1986.

\end{thebibliography}
		
\end{document}